\documentclass[a4paper,11pt]{article}
\usepackage{graphicx}
\usepackage{amsmath,amssymb,amsthm}

\newtheorem{theorem}{Theorem}[section]
\theoremstyle{plain}
\newtheorem{conjecture}[theorem]{Conjecture}
\newtheorem{lemma}[theorem]{Lemma}

\newtheorem{corollary}[theorem]{Corollary}
\newtheorem{problem}[theorem]{Problem}

\theoremstyle{definition}

\newtheorem{example}[theorem]{Example}

\begin{document}

\author{Klas Markstr\"om}

\title{Two questions of Erd\H{o}s on hypergraphs above the Tur{\'a}n threshold}

\maketitle

\begin{abstract}
	For ordinary graphs it is known that any graph $G$ with more edges than the Tur{\'a}n number of $K_s$ must contain several copies 
	of  $K_s$, and a copy of $K_{s+1}^-$, the complete graph on $s+1$ vertices with one missing edge. Erd\H{o}s asked if the 
	same result is true for $K^3_s$, the complete 3-uniform hypergraph on $s$ vertices.  
	
	In this note we show that for small values of $n$, the number of vertices in $G$, the answer is negative for $s=4$. For the second 
	property, that of containing a ${K^3_{s+1}}^-$, we show that for $s=4$ the answer is negative for all large $n$ as well, by proving that the 
	Tur{\'a}n density of  ${K^3_5}^-$ is greater than that of $K^3_4$. 
\end{abstract}

\section{Introduction}
One of the cornerstones of modern graph theory is Tur{\'a}n's theorem, which gives the maximum number of edges in a 
graph on $n$ vertices with no complete subgraphs of order $s$. This theorem has been extended in a number of different 
directions and we now have a good understanding of the corresponding question for other forbidden subgraphs, as long as 
they are not bipartite, and the structure of graphs close to the Tur{\'a}n threshold.

Recall that the Tur{\'a}n number $t(n,s,k)$ is the maximum number of edges in a $k$-uniform hypergraph on $n$ vertices 
which does not have the complete $k$-uniform hypergraph on $s$ vertices as a subgraph.  Rademacher proved, but did not 
publish, see \cite{EP62}  that any graph with $t(n,3,2)+1$ edges contains at least $n/2$ triangles. Erd\H{o}s later \cite{EP62} published a proof 
of this result and extended it by proving that for $0\leq q \leq c_1 n/2 $, for some constant $c_1$, any graph with $t(n,3,2)+q$ 
edges  contains at least $qn/2$ triangles. This result was later extended to graphs with density higher than $t(n,3,2)/{n \choose 2}$, 
recently culminating in \cite{Raz08} where the minimal number of triangles was determined for all densities.

Erd\H{os} and Dirac also observed that any graph with  $t(n,s,2)+1$ edges has $K_s^-$, 
the complete graph with one edge removed, as a subgraph. This provides a strengthening of the  result for triangles from \cite{EP62} in 
the sense that it shows that the graph contains two copies of $K_3$ which share two vertices.

In \cite{EP94} Erd\H{o}s asked if these results could be generalized to 3-uniform hypergraphs as well
\begin{problem}\label{pr1}
	Does every 3-uniform hypergraph on $t(n,s,3)+1$ edges contain two $K_s^3$?
\end{problem}
\begin{problem}\label{pr2}
	Does every 3-uniform hypergraph on $t(n,s,3)+1$ edges contain $K_{s+1}^{3-}$?
\end{problem}
In \cite{CG} these questions,  with an affirmative answer to both,  were formulated as conjectures,

The aim of this note is to show that the answer to both questions is no. For the first question we believe that the answer is yes for sufficiently 
large $n$, but as we shall see the second question fails even in an asymptotic sense.

\section{The examples}
We first construct a family of hypergraphs which give a negative answer to the two questions for certain small values of $n$. 
\begin{example}
	The $k$-uniform hypergraph $H_{k}$ is the hypergraph with vertex set $V=[2k-1]$ and edge set 
	$E={ [2k-1] \choose k}\setminus\{\{1,2,\ldots,k\},\{1,k+1,k+2,\ldots,2k-1\}\}$
\end{example}
We now have
\begin{theorem}
	Let $H_k$ be define as in the example.
	\begin{enumerate}
		\item $H_k$ does not contain two $K^k_{2k-2}$
		\item $|E(H_k)|\geq t(2k-1,2k-2,k)+1$
	\end{enumerate}
\end{theorem}
\begin{proof}
	\begin{enumerate}
		\item The vertex set of any $K^k_{2k-2}$ in $H_k$ cannot have one of the two non-edges as a subset. Let $A$ and $B$ be subsets of 
		$V(H_k)$ of size $2k-2$. Since any vertex subset of size $2k-2$ misses one vertex  of $H_k$ at least one of $A$ and $B$ must 
		have one of the two non-edges as a subset.
		
		\item Since every subset of $V(H_kj)$ of size $2k-2$ can have at most ${2k-2 \choose k}-1$ edges, we find by averaging that 
		\begin{equation} \label{eq1}
			t(2k-1,2k-2,k)\leq \frac{{2k-2 \choose k}-1}{{2k-2 \choose k-2}}{2k-1 \choose k} = {2k -1 \choose k } - \frac{2k-1}{k-1}
		\end{equation} 	
		The number of edges in $H_k$ is ${2k-1 \choose k}-2$ and Inequality \ref{eq1} shows that, since $\frac{2k-1}{k-1} > 2$, the Tur{\'a}n 
		number $t(2k-1,2k-2,k)$ is at most  ${2k-1 \choose k}-3$.  
				
	\end{enumerate}

\end{proof}
\begin{corollary}
	$H_3$ gives a negative answer to Problem \ref{pr1} and since there are two non-edges in $H_k$ it does not contain a $K_{5}^{3-}$ either, thereby giving 
	a negative answer to Problem \ref{pr2} as well.  
\end{corollary}
Note that $H_k$ also provides a negative answer to a generalisation of both problems to $k$-graphs for $k\geq 3$. However we believe that Question 1 
should have a positive answer for sufficiently large values of $n$
\begin{conjecture}
	There is an $n_0(s)$ such that every 3-uniform hypergraph on $n\geq n_0(s)$ vertices and  $t(n,s,3)+1$ edges contains two $K_s^3$
\end{conjecture}

For problem \ref{pr2} the failure is not a small $n$ phenomenon, as our next theorem will imply. Recall that the Tur{\'a}n density $\pi(G)$ of a $k$-uniform 
hypergraph $G$ is 
$$\pi(G)=\lim_{n\rightarrow \infty}\frac{t(n,G)}{{n \choose k}},$$
where $t(n,G)$ is the maximum number of edges in a $k$-uniform hypergraph on $n$ vertices which does not have $G$ as a subgraph.
We note that a positive answer to Problem \ref{pr2} requires that $\pi(K_s^3)=\pi(K_{s+1}^{3-})$. For $K_5^{3-}$ we have the following lower bound,
\begin{theorem}
	  $\pi(K_5^{3-}) \geq \frac{46}{81}$
\end{theorem}
\begin{proof}
	Let $H_3$ be the 3-uniform hypergraph defined in our earlier example and let $H_3(n)$, for $n$ divisible by 9, be the blow-up of $H_3$ with 
	$n/9$ copies of vertex 1 in $H_3$ and $2n/9$ copies of each of the other vertices in $H_3$. Three vertices $v_i,u_j,w_k$ form an edge in $H_3(n)$ if $\{v,u,w\}$ 
	is an edge in $H_3$. It is easy to see that $H_3(n)$ does not contain a $K_5^{3-}$ and a simple calculation shows that the number of edges in $H_3(n)$ 
	is  $(\frac{32}{81}+o(1)){n \choose 3}$
	
	Next, let $D$ be the directed graph on the same vertex set as  $H_3$ shown in Figure \ref{dir}.
	\begin{figure}[ht]
    		\begin{center}
			\includegraphics[width=0.5\textwidth]{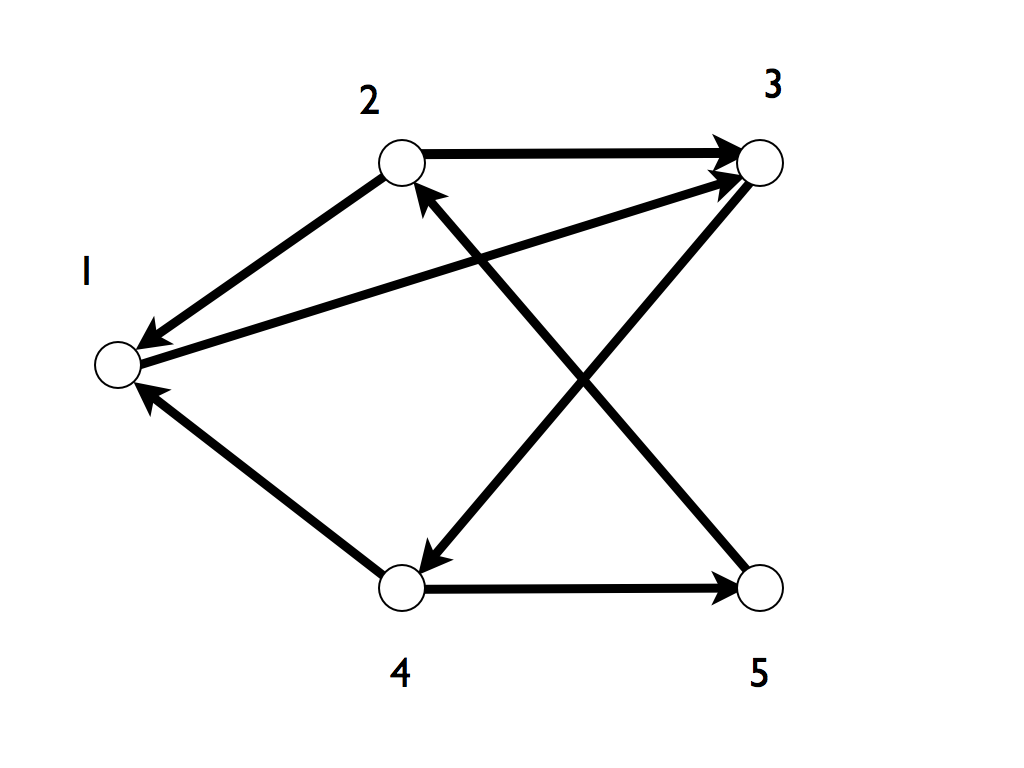}
			\caption{The auxiliary graph $D$ }\label{dir}
		\end{center}
	 \end{figure} 
	We now define a  3-uniform hypergraph $G_0(n)$ on the  same vertex set as $H_3(n)$ where  $\{v_i,v_j,u_k\}$ is an edge if, $i<j$ and 
	there is an edge from $v$ to $u$ in $D$. An edge in $D$ not incident with vertex 1 gives rise to $(\frac{8}{243}+o(1)){n \choose 3}$ edges. An edge leading to 
	vertex 1 gives rise to $(\frac{4}{243}+o(1)){n \choose 3}$ edges, and an edge leading from vertex 1 gives rise to $(\frac{2}{243}+o(1)){n \choose 3}$ edges.
	
	Finally we let $G(n)$ be the graph on the same vertex set as $H_3(n)$ with all edges from $H_3(n)$ and $G_0(n)$.  The number of edges in $G(n)$ is 
	\begin{multline}
		(\frac{32}{81}+o(1)){n \choose 3} + (\frac{2}{243}+o(1)){n \choose 3} +  2(\frac{4}{243}+o(1)){n \choose 3} +\\
		 4 (\frac{8}{243}+o(1)){n \choose 3}= (\frac{46}{81}+o(1)){n \choose 3}
	\end{multline}	
	In order to prove the theorem we now need to show that $G(n)$ does not contain $K_5^{3-}$.  
	
	Let $A$ be a set of 5 vertices in $G(n)$. If $A$ corresponds 
	to five distinct vertices in $H_3$ then $A$ is not a $K_5^{3-}$. If $A$ contains four or five vertices corresponding to the same vertex in $H_3$ then there are at 
	least 4 non-edges in $A$, and $A$ is not a  $K_5^{3-}$.  If $A$ has two vertices corresponding to the same vertex in $H_3$ and two other vertices corresponding 
	to another vertex in $H_3$ then there are at least 2 non-edges in $A$, and $A$ is not a  $K_5^{3-}$. 
	
	There are now two remaining possibilities:
	\begin{enumerate}
		\item There are three vertices in $A$ corresponding to the same vertex $v$ in $H_3$ and two vertices corresponding to two 
		other, distinct, vertices $u$ and $w$ in $D$. If there is at most one edge from $v$ to $u$ and $w$ in $D$ then there are at least 
		three non-edges in $A$, and $A$ is not a  $K_5^{3-}$. If there are two edges from $v$ to $u$ and $w$ then $\{v,u,w\}$ must be one of 
		the two non-edges in $H_3$, which means that there are at least three non-edges in $A$, and $A$ is not a  $K_5^{3-}$. 
		
		\item There are two vertices in $A$ corresponding to the same vertex $v$ in $H_3$ and three vertices corresponding to three 
		other, distinct, vertices $u$, $w$ and $z$ in $D$. If there are two edges from $v$ to $u,w$ and $z$ then $v$ together with two 
		of the other vertices are one of the non-edges in $H_3$ and there are at least two non-edges in $A$, and $A$ is not a  $K_5^{3-}$. 
		If there is only one edge from $v$ to the other three vertices there are at least two non-edges in $A$, and $A$ is not a  $K_5^{3-}$. 
	\end{enumerate}
	Hence $t(n,K_5^{3-})\geq  (\frac{46}{81}+o(1)){n \choose 3}$ and  $\pi(K_5^{3-}) \geq \frac{46}{81}$.

\end{proof}

Tur{\'a}n \cite{Tu61} posed the problem of finding the Tur{\'a}n density of the complete uniform hypergraphs and conjectured that $\pi(K^3_4)=\frac{5}{9}$, and 
gave a construction attaining this value. The conjecture remains open and the best current upper bound has been found by Razborov.  
\begin{lemma}\cite{Raz10}
	$\pi(K^3_4)\leq 0.561666$
\end{lemma}
The derivation of this bound in \cite{Raz10} is done by the flag-algebra method and is semi-numerical, but there is now freely available software \cite{BT11} 
which will give a computer assisted but non-numerical proof, based on the flag-algebra method.

Our lower bound for $\pi(K_5^{3-})$ is $ \frac{46}{81}=0.5679\ldots$, which is larger than the upper bound for $\pi(K_4^{3})$. Hence we get our desired corollary,
\begin{corollary}
	For $s=4$ Problem \ref{pr2} has a negative answer for all sufficiently large $n$ 
\end{corollary}
We believe that similar constructions will be possible for small values of $s$ and make the following conjecture,
\begin{conjecture}
	$\pi(K^3_s)<\pi(K_{s+1}^{3-})$, for $s=5,6$ 
\end{conjecture}
However, for large $s$ the two densities might coincide. In particular it would be interesting to settle the following question
\begin{problem}
	Is $\pi(K^3_7)=\pi(K_8^{3-})$?
\end{problem}


\begin{thebibliography}{Raz10}

\bibitem[BT11]{BT11}
Rahil Baber and John Talbot.
\newblock Hypergraphs do jump.
\newblock {\em Combin. Probab. Comput.}, 20(2):161--171, 2011.

\bibitem[CG98]{CG}
F.R.K. Chung and R.L. Graham.
\newblock {\em Erdos on graphs: His legacy of unsolved problems}.
\newblock AK Peters, 1998.

\bibitem[Erd62]{EP62}
P.~Erd{\H{o}}s.
\newblock On a theorem of {R}ademacher-{T}ur\'an.
\newblock {\em Illinois J. Math.}, 6:122--127, 1962.

\bibitem[Erd94]{EP94}
P.~Erd{\H{o}}s.
\newblock Problems and results on set systems and hypergraphs.
\newblock In {\em Extremal problems for finite sets ({V}isegr\'ad, 1991)},
  volume~3 of {\em Bolyai Soc. Math. Stud.}, pages 217--227. J\'anos Bolyai
  Math. Soc., Budapest, 1994.

\bibitem[Raz08]{Raz08}
Alexander~A. Razborov.
\newblock On the minimal density of triangles in graphs.
\newblock {\em Combin. Probab. Comput.}, 17(4):603--618, 2008.

\bibitem[Raz10]{Raz10}
Alexander~A. Razborov.
\newblock On 3-hypergraphs with forbidden 4-vertex configurations.
\newblock {\em SIAM J. Discrete Math.}, 24(3):946--963, 2010.

\bibitem[Tur61]{Tu61}
Paul Tur{\'a}n.
\newblock Research problems.
\newblock {\em Magyar Tud. Akad. Mat. Kutat\'o Int. K\"ozl.}, 6:417--423, 1961.

\end{thebibliography}

\end{document}